\documentclass[a4paper,11pt]{amsart}

\usepackage[french, english]{babel} 
\usepackage[T1]{fontenc}
\usepackage[ansinew]{inputenc}
\usepackage{graphicx,amsmath}

\date{}
\title[Observability estimates on the equilateral triangle]{Observability estimates for the Schr\"odinger equation on the equilateral triangle}

\author{Paul Alphonse}
\address{(Paul \textsc{Alphonse}) Institut de Math\'ematiques de Toulouse, UMR 5219, Universit\'e de Toulouse, CNRS, UPS, F-31062 Toulouse Cedex 9, France}
\email{paul.alphonse@math.univ-toulouse.fr}

\author{David Lafontaine}
\address{(David \textsc{Lafontaine}) CNRS and Institut de Math\'ematiques de Toulouse, UMR 5219, Universit\'e de Toulouse, CNRS, UPS, F-31062 Toulouse Cedex 9, France}
\email{david.lafontaine@math.univ-toulouse.fr}


\makeatletter
	\@namedef{subjclassname@2020}{\textup{2020} Mathematics Subject Classification}
\makeatother


\usepackage[top=3cm, bottom=2cm, left=3cm, right=3cm]{geometry}

\usepackage{enumitem}

\usepackage{array}										

\usepackage{amsfonts}
\usepackage{amsmath}
\usepackage{amsthm}
\usepackage{amssymb}
\usepackage{xfrac}
\usepackage{bbm}

\usepackage{stmaryrd}

\usepackage{ulem}

\usepackage[scr]{rsfso}

\usepackage{comment}

\usepackage[dvipsnames]{xcolor}

\usepackage{hyperref}

\usepackage{tikz}
\usetikzlibrary{calc} 
\usetikzlibrary {arrows.meta}

\hypersetup{	
colorlinks=true,
breaklinks=true,
urlcolor= RedViolet,
linkcolor= red,
citecolor=Blue
}

\numberwithin{equation}{section}
										
\newtheorem{thm}{Theorem}[section]
\newtheorem{prop}[thm]{Proposition}

\newtheorem{lem}[thm]{Lemma}
\newtheorem{cor}[thm]{Corollary}

\newtheorem*{dfn*}{Definition 1.0}
\theoremstyle{definition}
\newtheorem{ex}[thm]{Example}
\newtheorem{rk}[thm]{Remark}

\DeclareMathOperator{\Supp}{Supp}

\DeclareMathOperator{\var}{\mathcal T}
\DeclareMathOperator{\tori}{\mathbb T^2[\omega]}
\DeclareMathOperator{\supp}{supp}
\DeclareMathOperator{\divv}{div}
\DeclareMathOperator{\GL}{GL}

\newcommand{\twotorus}{\mathbb T^2_A}

\begin{document}


\begin{abstract} 
We prove observability estimates for the Schr\"odinger equation posed on the equilateral triangle in the plane, under both Neumann and Dirichlet boundary conditions. No geometric control condition is required on the rough localization functions that we consider. This is the first result of this kind on a non-toric domain in the compact setting. Our strategy is to exploit Pinsky's tiling argument to deduce this result from observability estimates on rational twisted tori. These are obtained via propagation of singularities, adapting arguments from Burq and Zworski. The later require Strichartz estimates on such twisted rational tori, that we derive from Zygmund inequalities in the same geometric setting, also providing the sharp constant. Strichartz estimates on the equilateral triangle are also derived from this analysis.
\end{abstract}

\sloppy

\selectlanguage{english}

\maketitle

\section{Introduction and main results}

\subsection{Context and main result}

This work is part of the study of controllability of Schr\"odinger equations of the form
\begin{equation}\label{eq:genschro}\tag{$S_{\mathcal M}$}
	\left\{\begin{aligned}
		& i\partial_t u(t,x) + \Delta u(t,x) = a(x)\mathbbm 1_{(0,T)}h(t,x),  & (t,x)\in\mathbb R\times\mathcal M, \\
		& u(0,\cdot) = u_0\in L^2(\mathcal M),
	\end{aligned}\right.
\end{equation}
where $\mathcal M$ is a compact Riemannian manifold, $\Delta$ is the Laplace--Beltrami operator on $\mathcal M$, with Dirichlet or Neumann boundary condition when $\partial M\ne\emptyset$, \textit{i.e.}
\[
    u(t,x) = 0\quad \text{or}\quad \partial_nu(t,x) = 0\quad \text{on $\mathbb R\times\partial M$}, 
\]
and $a\in L^p(\mathcal M)$ is a localization function for some $p\in[1,+\infty]$. A central problem is to characterize those $a\in L^p(\mathcal M)$ for which controllability holds for the equation \eqref{eq:genschro}, \textit{i.e.}~for which, given $u_0, u_1 \in L^2(\mathcal M)$, there exists a control  $h$ such that the solution to \eqref{eq:genschro} satisfies $u(T,\cdot) = u_1$. This question is deeply related to the geometric properties of the manifold $\mathcal M$. In this very general setting, it is now well-known that when $a = \mathbbm1_{\omega}$ is the indicator function of some open set $\omega\subset\mathcal M$, the equation \eqref{eq:genschro} is controllable at any positive time $T>0$ provided $\omega$ verifies the so-called geometric control condition. The latter was introduced in \cite{BLR, RT} for the wave equation, and was first related to the Schr\"odinger equation in the seminal work \cite{L} in the context of boundary control, \textit{i.e.}~with a localization function $a$ supported on the boundary, (see also \cite{Ma, P} for internal results).

While the geometric control condition is necessary and sufficient to ensure the controllability of the wave equation \cite{BG}, this is not the case in general for the Schr\"odinger equation. This sufficient geometric condition is necessary in some cases, as for the sphere $\mathcal M = \mathbb S^2$ (for which some spherical harmonics concentrate around the equator) or the disk $\mathcal M = \mathbb D$ (the control support $\omega$ has to intersect the boundary, due to the presence of whispering-gallery modes) \cite{ALM}, and not necessary in some others, such as for the tori $\mathcal M = \mathbb T^d$ with $d\geq1$ \cite{Jaf,K}, or  hyperbolic compact surfaces (more generally, Anosov surfaces) \cite{DJN, J}, where the control support $\omega\subset\mathcal M$ may be taken to be any nonempty subset. 

In the case of one-- and two-- dimensional tori, one can even prove controllability with rough localization functions $a\in L^p(\mathbb T^d)$, with $p=1$ and $p=2$ respectively; see, for instance, \cite{AT, BZ, LBM, NWX}. This includes the interesting case where $a = \mathbbm1_{\omega}$ and $\omega\subset\mathbb T^d$ is a measurable set with positive measure. Extending these results to higher dimensional tori $d\geq3$ remains a challenging open problem. Moreover, these are the only two results currently available in the setting of rough control. 

In this paper, we investigate a new planar geometric context and prove a controllability result with rough localization functions and no geometric control condition in a non-rectangular setting. More precisely, we consider the equilateral triangle
\[
    \var = \big\{(x_1,x_2)\in\mathbb R^2 : 0<x_1<1,\,0<x_2<x_1\sqrt{3},\,x_2<\sqrt{3}(1-x_1)\big\}.
\]
Our controllability result is the following.

\begin{thm} Let $T>0$ be a positive time and $a\in L^{\infty}(\var)\setminus\{0\}$ be a non-negative function. Then, for every initial datum $u_0, u_1\in L^2(\var)$, there exists a control $h\in L^2([0,T]\times\var, \mathbb C)$ such that the solution $u$ of the equation {\rm (}\hyperref[eq:genschro]{$S_{\var}$}{\rm )} satisfies $u(T,\cdot)=u_1$.
\end{thm}

To the best of our knowledge, this is the first controllability result for the Schr\"odinger equation with no geometric control condition and rough localization functions on a non-rectangle manifold with boundary. Let us mention that controllability results with boundary controls for the equation (\hyperref[eq:genschro]{$S_{\var}$}) posed on general triangles (actually, on general simplices) are nevertheless stated in \cite[Theorem 2]{CC}.

Due to time reversibility (after conjugation), the notions of controllability and null-controllability (\textit{i.e.} controllability to $u_1 = 0$) are equivalent for the equation \eqref{eq:genschro}. Moreover, according to the Hilbert Uniqueness Method, proving null-controllability for the equation \eqref{eq:genschro} is equivalent to proving the following observability estimate: for every initial datum $u_0\in L^2(\mathcal M)$,
\begin{equation}\label{eq:genobs}
	\Vert u_0\Vert^2_{L^2(\mathcal M)}\le C_{a,T}\int_0^T\int_{\mathcal M}a(x)\big\vert(e^{it\Delta}u_0)(x)\big\vert^2\,\mathrm dx\,\mathrm dt.
\end{equation}
From now on, we will no longer refer to the controllability property itself, but only to the corresponding observability estimate. This estimate, which is the main result of this paper, is given in the following theorem,  where $(e^{it\Delta})_{t\in \mathbb R}$ refers to the Schr\"odinger flow on $\mathcal T$ with Dirichlet or Neumann boundary condition.

\begin{thm}\label{thm:obstriangle} 
Let $T>0$ be a positive time and $a\in L^2(\var)\setminus\{0\}$ be a non-negative function. There exists a positive constant $C_{a,T}>0$ such that for every initial datum $u_0\in L^2(\mathcal T)$,
\[
	\Vert u_0\Vert^2_{L^2(\var)}\le C_{a,T}\int_0^T\int_{\var}a(x)\big\vert(e^{it\Delta}u_0)(x)\big\vert^2\,\mathrm dx\,\mathrm dt.
\]
\end{thm}

\begin{rk}\label{rk:stritri}
While proving Theorem \ref{thm:obstriangle}, we also establish Strichartz estimates on the triangle $\var$, of independent interest: for every solution $u$ of the equation
\[
    \begin{cases}
	    i\partial_t u(t,x) + \Delta u(t,x) = F(t,x),\quad (t,x)\in\mathbb R\times\var, \\
    	u(0,\cdot) = u_0\in L^2(\var).
    \end{cases}
\]
the following estimate holds
\begin{equation}\label{eq:inhomostritri}
	\Vert u\Vert_{L^4(\var,L^2(0, \frac{27}{8\pi}))}
	\le\bigg(\frac{81\sqrt{3}}{32\pi^2}\bigg)^{\frac 14}\big(\Vert u_0\Vert_{L^2(\var)}+ \Vert F\Vert_{L^1((0,\frac{27}{8\pi}),L^2(\var))}\big).
\end{equation}
In the homogeneous case where $F=0$, this estimate ensures the finitness of the right-hand side of the observability estimate stated in Theorem \ref{thm:obstriangle}, and forces to consider rough localization functions $a\in L^2(\var)$. Moreover, the inhomogeneous estimate is key to implement the Hilbert Uniqueness Method for which \cite[Proposition 4.1]{BZ} holds \textit{mutatis mutandis} provided $a\in L^{\infty}(\var)$ and $h \in L^2([0,T]\times\mathcal T)$. The reason why we cannot implement the Hilbert Uniqueness Method when $a\in L^2(\var)$ (the localization functions considered in Theorem \ref{thm:obstriangle}) is that we do not have at hand the required $L^{4/3}_xL^2_t$-inhomogeneous Strichartz used in \cite[Proposition 4.1]{BZ}.
\end{rk}

\subsection{Strategy of the proof} 
The works \cite{BBZ, BZ} provide a clear strategy for proving observability estimates of the form \eqref{eq:genobs} with rough localization functions on tori, and it would be natural to attempt to extend this approach to the triangle $\mathcal T$. However, since this strategy relies on semiclassical propagation of singularities, its implementation in our context is \textit{a priori} complicated to perform, due both to the presence of corners and to the more intricate nature of the geodesic flow on $\mathcal T$. Yet, we will be able to use a tiling argument due to Pinsky \cite{P2} to derive observability estimates on the triangle $\mathcal T$ from observability estimates on an explicit rational twisted torus, taking the form $\mathbb R^2/A\mathbb Z^2$ with $A\in\GL_2(\mathbb R)$ non symmetric. The fundamental reason for this trick comes from a previous result by Lam\'e \cite{La}, stating that the (Dirichlet and Neumann) eigenfunctions of the Laplacian $-\Delta$ acting on $L^2(\var)$ enjoy symmetry properties with respect to the sides of the triangle, see Figure \hyperref[fig:sym]{1}. Observability estimates on general rational twisted tori of the form $\mathbb R^2/A\mathbb Z^2$  are in turn obtained by adapting the arguments from \cite{BZ}. However, the authors of \cite{BZ} consider only rectangular tori, and neither their results nor their proofs apply directly to our situation. It is therefore necessary to supplement their arguments to address the case of twisted tori. Note that this new argument is valid only for rational twisted tori (which is sufficient to treat the case of the equilateral triangle we are interested in), and the case of the irrational twisted tori remains open. In addition, and of independent interest, we generalize Zygmund inequality, which is required to set up the propagation argument, to twisted tori, taking this opportunity to provide the sharp constant in the rectangular case.

Let us illustrate the tiling argument used to go from the triangle $\mathcal T$ to a twisted torus on the (easy) case of the square with Neumann boundary conditions.

\begin{figure}\label{fig:sym}
    \includegraphics{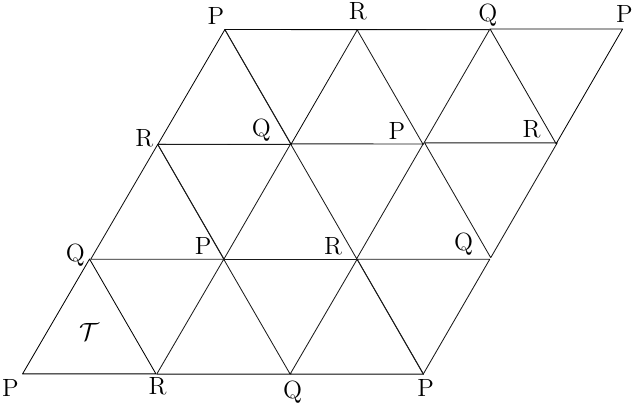}
    \caption{From the triangle $\mathcal T$ to the twisted torus $\mathbb R^2/A\mathbb Z^2$ by successive reflections.}
\end{figure}

\begin{ex} \label{ex:rec} Let us consider the square $\Omega = [0,\pi]\times[0,\pi]$. We recall that the Neumann eigenvalues of the Laplacian $-\Delta$ acting on $L^2(\Omega)$ are given by $\lambda_m = -\vert m\vert^2$, $m\in\mathbb Z^2$, with the following associated eigenfunctions
\[
    f_m(x_1,x_2) = \frac2{\pi}\cos(m_1x_1)\cos(m_1x_2),\quad m=(m_1,m_2)\in\mathbb Z^2,\,(x_1,x_2)\in\Omega.
\]
By extending the eigenfunctions $f_m$ by reflections with respect to the left and the upper sides of the square $\Omega$, one can see that proving an observability estimate for the Schr\"odinger equation on $\Omega$ is equivalent to proving an observability estimate on the torus $\mathbb R^2/(2\pi\mathbb Z)^2$, the latter being available in \cite{BZ}. Naturally, a similar discussion can be held for the Dirichlet eigenfunctions of $-\Delta$, with antireflections instead of reflections, since they are given by products of sines. The very same strategy will be used for the triangle $\mathcal T$, with more sophisticated arguments.
\end{ex}

\subsection{Comments and generalizations}

Some remarks regarding possible generalizations of Theorem \ref{thm:obstriangle} are in order.
 
A natural question is to wonder to which class of polygonal domains $\Omega\subset\mathbb R^2$ our method applies. If  $\Omega\subset\mathbb R^2$ is such a domain, replicating the tiling argument from Pinsky \cite{P2} shows that it possesses a complete set of trigonometric eigenfunctions (\textit{i.e.}, given by finite trigonometric sums). A result of McCartin \cite{MC} shows that the only such polygonal domains are the rectangle (see Example \ref{ex:rec}), the isosceles right triangle, the equilateral triangle and the hemiequilateral triangle. A slight modification of our arguments gives the analog to Theorem \ref{thm:obstriangle} for these domains, while the problem remains open for any other polygonal domain.

Many results in the literature deal not only with the observability properties of the free Schr\"odinger equation, \textit{i.e.}~of the equation \eqref{eq:genschro}, but also of the same equation perturbed by a bounded zero-order potential, see e.g.~\cite{AFM, AM, Bou, BBZ, LBM, Ma}. Among these works, only \cite{LBM} considers both rough localization functions and such perturbations, while the others restrict to indicator functions of open sets. It would be very interesting to extend our result to such perturbations, and to prove observability estimates of the form \eqref{eq:genobs} for the operator $\Delta+V$, where $V\in L^{\infty}(\var)$,  on the equilateral triangle $\var$. One of the supplementary ingredients used in  \cite{BZ} to deal with $V$ is a normal form argument  performed  in the reduction dimension. Providing such an argument in the context of twisted tori (which would permit to handle the case of the triangle $\mathcal T$) remains open.

Finally, in Theorem \ref{thm:obstriangle}, we consider localization functions $a\in L^2(\var)$ not depending on time, and a natural extension would be to deal with the case where $a = a(t,x)$. This setting includes the particularly interesting example $a(t,x) = \mathbbm1_{\omega(t)}$, where $\omega(t)\subset\var$ are moving measurable control supports of positive measure. There are actually very few results in the literature dealing with this situation. Let us mention the recent paper \cite{NWX} where the authors prove observability estimates for the Schr\"odinger equation on the circle $\mathbb T$ with time-dependent rough localization functions. 

\medskip

\subsubsection*{Structure of the paper}
Section \ref{sec:til} is devoted to present in details the tiling argument that allows us to derive observability estimates on the triangle $\mathcal T$ from observability estimates on a rational twisted torus. In Section \ref{ss:diamond}, we establish an extension of Zygmund inequality to general twisted tori, that moreover provides the sharp constant in the square case. This result plays a key role while establishing observability estimates on general rational twisted tori. Such observability estimates 
are finally obtained in Section~\ref{sec:obstori}. 

\section{Estimates on equilateral triangles through tiling arguments} \label{sec:til}

In this section, we explain how tiling arguments allow to deduce estimates for the Schr\"odinger equation on the equilateral triangle $\mathcal T$ from estimates on a rational twisted torus.

\subsection{Spectral decomposition}
First, we need to recall some spectral properties of the Dirichlet and the Neumann Laplacian on the equilateral triangle. It is known from the works of Lam\'e \cite{La} and Pinsky \cite{P2, P1} that the eigenvalues of $-\Delta$ on $\mathcal T$ are given by
\begin{equation}\label{eq:eigentri}
    \lambda_{m,n} = \frac{16 \pi^2}{27}\big(m^2 + n^2 - mn\big),
\end{equation}
with $m,n \in \mathbb Z$, where $3\,|\,m+ n$ in the case of Neumann boundary conditions, and $3\,|\,m+ n$, $m\neq 2n$, $n\neq 2m$ in the case of Dirichlet boundary conditions. The above works also give a complete description of the associated eigenfunctions, that we introduce now. To that end, let
$$
	k_{m,n} := \frac{2\pi}{3}\bigg(m, \frac{2n - m}{\sqrt 3}\bigg),
$$
and, for $\xi \in \mathbb R^2$,
$$
	e_{\xi}(x) := \exp({i \xi \cdot x}),\quad x\in \mathbb R^2,
$$
where $\cdot$ denotes the Euclidean scalar product in $\mathbb R^2$. In addition, let $G$ be the subgroup of $\GL_2(\mathbb R)$ generated by the matrices
        \[
            R_1 := \begin{pmatrix} 1 & 0 \\ 0 & -1\end{pmatrix}\quad\text{and}\quad
            R_2 := \begin{pmatrix}-\frac12 & \frac{\sqrt 3}{2} \\ \frac{\sqrt 3}{2} & \frac12\end{pmatrix},
        \]
which are the orthogonal reflections by $\mathbb R e$ and $\mathbb R \omega$ respectively, where
$$
e := (1,0) \quad\text{and}\quad \omega := \bigg(\frac12,\frac{\sqrt 3}{2}\bigg) = \exp\bigg(\frac{i\pi}{3}\bigg).
$$
Finally, let $\varepsilon : G \to \{ -1,1\}$ be the unique group homomorphism such that
$$
\varepsilon(I_2) = 1, \quad \varepsilon(R_1) = -1, \quad \varepsilon(R_2) = -1.
$$
In other words,
$$
G = \big\{I_2,R_1,R_2R_1,R_1R_2R_1,R_1R_2,R_2 \big\},
$$
and $\varepsilon$ is alternating in the above list. The $L^2$-normalized Neumann and Dirichlet eigenfunctions for the eigenvalue $\lambda_{m,n}$ are respectively given for $x\in\mathcal T$ by
\begin{equation} \label{eq:ef}
        u_{m,n}(x) = \kappa_{m,n}^N \sum_{g\in G}e_{gk_{m,n}}(x)\quad\text{and}\quad v_{m,n}(x) = \kappa_{m,n}^D \sum_{g\in G}\varepsilon(g)e_{gk_{m,n}}(x),
\end{equation}
where $\kappa_{m,n}^N>0$ and $\kappa_{m,n}^D$ are renormalization factors. We make choices of sign in the unsigned renormalization factor $\kappa_{m,n}^D$ in such a way that if $G k_{m_1, n_1} = G k_{m_2, n_2}$, then $v_{m_1, n_1} = v_{m_2, n_2}$. Pinsky showed in \cite{P2} that these systems are indeed complete in $L^2(\var)$. Hence, every function $f \in L^2(\mathcal T)$ decomposes as
\[
    f = \sum_{\substack{(m, n) \in \mathbb Z^2/\sim \\ 3 \,|\,m+n}}  a_{m,n}(f) u_{m,n} 
    = \sum_{\substack{(m, n) \in \mathbb Z^2/\sim \\ 3 \,|\,m+n \\ n\neq 2m,\, m\neq 2n}}  b_{m,n}(f) v_{m,n},
\]
where we defined
\[
    (m_1, n_1) \sim (m_2, n_2) \iff G k_{m_1, n_1} = G k_{m_2, n_2}.
\]
Observe that, using the explicit expressions of $R_1$ and $R_2$ above, we get
\begin{equation} \label{eq:SG}
    R_1 k_{m,n} = k_{S_1(m,n)} \quad\text{and}\quad R_2 k_{m,n} = k_{S_2(m,n)}, 
\end{equation}
where
$$
S_1 : (m,n) \to (m, m-n), \quad S_2 : (m,n) \to (n-m, n),
$$
and hence, in particular, the properties that $3\,|\,m+n$, $m\ne 2n$ and $n\ne 2m$ do not depend on the choice of representative of the equivalence class.

For the interested reader, we finish this subsection by briefly summing up the proof of density of these families, which is the object of Pinsky's paper \cite{P2}. An arbitrary function $f\in L^2(\mathcal T)$ is extended to the torus $\mathbb T^2[\omega] := \mathbb R^2 / 3\mathbb Z[\omega]$ by successives reflexions (see Figure~\ref{fig:sym}), in an odd way to recover Dirichlet eigenfunctions, or without sign change to recover Neumann eigenfunctions. This torus is then straightened by a change of variable $\varphi$, in such a way that $f\circ \varphi^{-1}$ is now an element of $L^2(\mathbb T^2$). Thus it can be decomposed as its usual Fourier decomposition. Undoing the change of variable then gives the decomposition of $f$ on $\mathcal T$ in terms of Neumann (resp. Dirichlet) eigenfunctions, described above. 

Finally, we give in Appendix \ref{app:ortho} an original (to our knowledge) proof of the orthogonality of the families of eigenfunctions $(u_{m,n})_{(m,n)}$ and $(v_{m,n})_{(m,n)}$.

\begin{rk} The families $(u_{m,n})_{(m,n)}$ and $(v_{m,n})_{(m,n)}$ are introduced in \cite{P2} as 
$$
u_{m,n} := \kappa^N_{m,n} \sum_{(m',n')\in S_{m,n}} e_{m',n'}\quad\text{and}\quad
v_{m,n} := \kappa^D_{m,n} \sum_{(m',n')\in S_{m,n}} \epsilon_{m,n}(m',n') e_{m',n'},
$$
where $e_{m,n} := e_{k_{m,n}}$, the superset (an element can appear more than once) $S_{m,n}$ is the orbit of $(m,n)$ under the action of the transformations $S_1$ and $S_2$, $\epsilon$ is the unique group homomorphism so that
$$
\epsilon(\operatorname{I}) = 1, \quad \epsilon(S_1) = -1,\quad \epsilon(S_2) = -1,
$$
and 
$$
\epsilon_{m,n}(m', n') := \epsilon(h) \; \text{ for } \; (m', n') = h(m,n) \in S_{m,n}.
$$
From (\ref{eq:SG}), this definition of $u_{n,m}$ and $v_{n,m}$ coincides with the definition (\ref{eq:ef}).
\end{rk}

\subsection{From equilateral triangle to twisted torus}\label{subsec:tritotor}

Let us explain how to reduce our problem to an observability estimate on the twisted torus
\begin{equation}\label{eq:partitorus}
    \mathbb T^2[\omega] := \mathbb R^2 / 3\mathbb Z[\omega],
\end{equation}
where $\mathbb Z[\omega]\subset\mathbb R^2$ is defined by
\[
    \mathbb Z[\omega] = \big\{a+\omega b : a,b\in\mathbb Z\big\}.
\]
Let us also denote by ${\mathcal H}_1 \subset \mathbb R^2$ the hexagon, orbit of $\mathcal T$ under the action of $G$
\begin{equation}\label{eq:h1}
	{\mathcal H}_1 := \bigcup_{g \in G} g \mathcal T,
\end{equation}
and let further
$$
{\mathcal H}_2 := {\mathcal H}_1 + e + \omega\quad\text{and}\quad {\mathcal H}_3 := {\mathcal H}_1 + 2e + 2\omega.
$$
Introducing now the Pinsky parallelogram
$$
\mathcal P := \big\{ a + \omega b : 0<a,b<3 \big\},
$$
which is a fundamental domain of the torus $\mathbb T^2[\omega]$, observe that
\begin{equation}\label{eq:pH}
    \mathcal P = {\mathcal H}_1 \cup {\mathcal H}_2 \cup {\mathcal H}_3 \quad\text{ in }\mathbb R^2 \slash 3\mathbb Z[\omega],
\end{equation}
as illustrated in Figure \hyperref[fig:pinsky]{2}. 

\begin{figure}\label{fig:pinsky}
\includegraphics{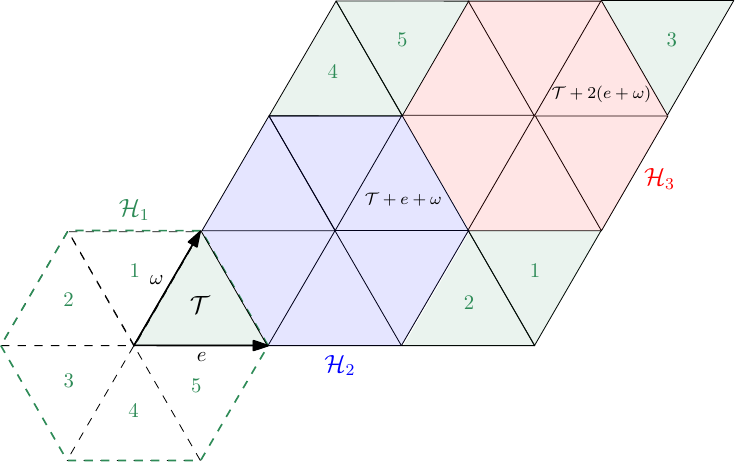}
 \caption{The parallelogram $\mathcal P$ as $18$ copies of $\mathcal T$ and its reflections, and its decomposition in the three hexagons $\mathcal {\mathcal H}_1$, $\mathcal {\mathcal H}_2$, $\mathcal {\mathcal H}_3$. The triangles with the same number are the same in $\mathbb R^2 / 3\mathbb Z[\omega]$.}
\end{figure}

Recall that $u_{m,n}$ and $v_{m,n}$ are a priori defined only on $\mathcal T$. We extend them to $\mathbb R^2$, and still denote them in the same way, by defining  $u_{m,n}(x)$ and $v_{m,n}(x)$ by (\ref{eq:ef}) also for $x \in \mathbb R^2$. We begin with a lemma quantifying the periodicity of these functions. 

\begin{lem} \label{lem:inv_uv}
Let $g \in G$, $p\in \mathbb Z$, and $m, n \in \mathbb Z$ so that $3 \, | \, m+n$.  We have for every $x\in\var$
$$
u_{m,n}(x) = u_{m,n}( gx+p(e+\omega)), \quad v_{m,n}(x) = \varepsilon(g) v_{m,n}( gx+p(e+\omega)).
$$
\end{lem}

\begin{proof}
Recall that
$$
u_{m,n}(x) = \kappa^N_{m,n} \sum_{g' \in G} \exp{(ig'k_{m,n}\cdot x)},
$$
hence, by a change of variable in the sum
\begin{align*}
	u_{m,n}(gx) &= \kappa^N_{m,n} \sum_{g' \in G} \exp{(ig'k_{m,n}\cdot gx)} 
	= \kappa^N_{m,n} \sum_{g' \in G} \exp{(i g^*g'k_{m,n}\cdot x}) \\ 
	& =  \kappa^N_{m,n} \sum_{g' \in g^*G} \exp{(ig'k_{m,n}\cdot x)}.
\end{align*}
But $g^*$ belongs to the group generated by $R_1^*$ and $R_2^*$. As $R_1$ and $R_2$ are self-adjoint, this group is nothing but $G$ itself, and hence $g^* \in G$. It follows that $g^* G = G$, and thus
$$
u_{m,n}(gx) = \kappa^N_{m,n} \sum_{g' \in G} \exp{(ig'k_{m,n}\cdot x)} = u_{m,n}(x),
$$
and similarly for $v_{m,n}$. To conclude, we observe that the $e_{g'k_{m,n}}$ are $(e+\omega)$--periodic for $g'\in G$ and $3 \, | \, m+n$. Indeed, this comes from the fact that $g'k_{m,n} = k_{m',n'}$ with $3 \, | \, m'+n'$ thanks to (\ref{eq:SG}), together with the property that  ${k_{m',n'} \cdot (ae + b \omega) = \frac{2\pi}{3}(m'a + n'b)}$.
\end{proof}

Let us define the following extension operators
$$
E^N f := \sum_{\substack{(m, n) \in \mathbb Z^2/\sim \\ 3 \,|\,m+n}}  a_{m,n}(f) u_{m,n}, \quad
E^D f := \sum_{\substack{(m, n) \in \mathbb Z^2/\sim \\ 3 \,|\,m+n \\ n\neq 2m,\, m\neq 2n}}  b_{m,n}(f) v_{m,n},\quad
f \in L^2(\mathcal T),
$$
where we recall that $u_{m,n}(x)$ and $v_{m,n}(x)$ are defined for $x \in \mathbb R^2$ by (\ref{eq:ef}). The operators $E^{N,D}$ satisfy the following properties.

\begin{prop} \label{prop:TtoTor}
For every $f \in L^2(\mathcal T)$, the functions $E^{N,D}f$ are well-defined as functions in $L^2(\mathbb T^2[\omega])$. Moreover, we have
$$
E^{N,D} \Delta_{N,D} f = \Delta E^{N,D} f,
$$
and in addition, 
$$
\forall g \in G, \; p \in \mathbb Z, \; \text{a.e.} \; x \in \mathcal T, \quad 
\begin{cases}
(E^{N}f)(g x + p(e+\omega)) = f(x), \\[5pt]
(E^{D}f)(g x + p(e+\omega)) = \varepsilon(g) f(x).
\end{cases}
$$
\end{prop}

\begin{proof}
First, observe that $E^{N,D} f$ are well-defined as functions on $\tori$, because the $e_{gk_{m,n}}$, and hence each $u_{m,n}$ (resp $v_{m,n}$) are $3e$ and $3\omega$ periodic. To get the first claim, it suffices to check that, for $3 \, | \, m+n$,
$$
\Delta u_{m,n}(x) = \lambda_{m,n}  u_{m,n}(x), \quad x \in \mathbb R^2,
$$
(and similarly for $v_{m,n}$) which follows from the fact that this is verified in $\mathcal T$, together with Lemma \ref{lem:inv_uv} and the fact that the Laplacian commutes with reflexions and translations. The second claim follows directly from Lemma \ref{lem:inv_uv}.
\end{proof}

Recalling the decomposition
$$
\mathcal P = \mathcal H_1 \cup \mathcal H_2 \cup \mathcal H_3 = \bigcup_{\substack{g \in G, \\ p =0,1,2}} g\mathcal T + p(e + \omega),
\quad\text{ in }\mathbb R^2 \slash 3\mathbb Z[\omega].
$$
we directly deduce from Proposition \ref{prop:TtoTor} the following.

\begin{cor} \label{cor:red_twist}
For any $f\in L^2(\var)$ and $p\geq 1$
$$
\Vert E^{N,D}f \Vert^p_{L^p(\tori)} = 18 \Vert f \Vert^p_{L^p(\mathcal T)},
$$
and for any $b \in L^2(\mathcal T)$ such that $b \geq 0$,
$$
\int_{\tori} |E^{N,D} b| |E^{N,D} f | ^2 = 18 \int_{\mathcal T} b |f|^2.
$$
\end{cor}

As a consequence of Corollary \ref{cor:red_twist}, proving the observability estimates from Theorem~\ref{thm:obstriangle} on the equilateral triangle $\mathcal T$ is equivalent to proving the following estimates on $\mathbb T^2[\omega]$ for every positive time $T>0$, $a\in E^{N,D}(L^2(\var))$ and $u_0\in E^{N,D}(L^2(\var))$,
\begin{equation}\label{eq:twistedobs}
    \Vert u_0\Vert_{L^2(\tori)}\le C_{a,T}\int_0^T\int_{\tori}
    a(x)\big\vert(e^{it\Delta}u_0)(x)\big\vert^2\,\mathrm dx\,\mathrm dt.
\end{equation}
In Section \ref{sec:obstori}, we will even be able to prove these observability estimates for every rough control $a\in L^2(\tori)$ and for every initial datum $u_0\in L^2(\mathbb T^2[\omega])$. In fact, we will establish \eqref{eq:twistedobs} on very general rational twisted tori (defined in the next section).

Naturally, this strategy is not specific to observability estimates, and can be applied for various kind of estimates, as the Strichartz estimate stated in Remark \ref{rk:stritri} (see Remark~\ref{rk:striinhomo}), and for other equations involving the Laplace operator.

\section{Sharp Zygmund inequality} \label{ss:diamond}

In the seminal work \cite{Z}, answering a question by Fefferman, Zygmund proved the following $L^4$ inequality for functions $u\in L^2(\mathbb T^2)$ with Fourier transforms $\widehat u$ supported in circles $\mathcal S_{\lambda}\subset\mathbb Z^2$,
\begin{equation}\label{eq:zygnonopti}
    \Vert u\Vert^4_{L^4(\mathbb T^2)}\le C\Vert u\Vert^4_{L^2(\mathbb T^2)},
\end{equation}
where $\mathbb T^2 = \mathbb R^2/\mathbb Z^2$, and the circles $\mathcal S_{\lambda}$ are defined for every $\lambda>0$ by
\[
    \mathcal S_{\lambda} := \big\{m\in\mathbb Z^2 : \vert m\vert^2 = \lambda\big\}.
\]
This kind of estimate, fundamental in harmonic analysis, also turns out to play a key role in controllability theory of dispersive equations, while proving observability estimates like \eqref{eq:twistedobs}. This will be detailed in Section \ref{sec:obstori}. Zygmund's original argument yields the inequality \eqref{eq:zygnonopti} with $C=5$ \cite[Theorem 1]{Z}. The aim of this section is two-fold: we extend \eqref{eq:zygnonopti} to twisted tori (which generalize $\tori$ and will be defined below), while giving another proof of \eqref{eq:zygnonopti} that yields the sharp constant $C=3$. We now define general twisted tori. Given $A\in \GL_2(\mathbb R)$, we associate to $A$ the following torus
\[
    \twotorus := \mathbb R^2 /(A\mathbb Z^2).
\]
Recall that the spectrum of the Laplace operator $-\Delta$ acting on functions in $L^2(\mathbb T^2_A)$ is composed of the following eigenvalues
\begin{equation}\label{eq:eigenvalues}
    \lambda_m = \vert 2\pi m\vert^2,\quad m\in\mathbb Z^2_A:=(A^{-1})^*\mathbb Z^2,
\end{equation}
where $(A^{-1})^*$ denotes the transpose of $A^{-1}$ and that each eigenvalue $\lambda_m>0$ is associated with the following normalized eigenfunction
\[
    e_m(x) = \frac{1}{\sqrt{\det A}}\,e^{2i\pi m\cdot x},\quad m\in\mathbb Z^2_A,\, x\in\twotorus.
\]
In this context, the natural energy levels are the sets $\mathcal S_{\lambda}(A)$ defined for every $\lambda>0$ by (we get rid of the factor $4\pi^2$, as for the circles $\mathcal S_{\lambda}$)
\[
    \mathcal S_{\lambda}(A) := \big\{m\in\mathbb Z^2_A : \vert m\vert^2 = \lambda\big\}.
\]
The generalization of Zygmund inequality \eqref{eq:zygnonopti} is stated in the following result.

\begin{thm}\label{thm:zyg} For every energy level $\lambda>0$ and for every function $u\in L^2(\twotorus)$ such that $\supp\widehat{u}\subset\mathcal S_{\lambda}(A)$, the following estimate holds
\[
    \Vert u\Vert^4_{L^4(\twotorus)}\le\frac{3}{\det A}\Vert u\Vert^4_{L^2(\twotorus)}.
\]
Moreover, the above constant $3$ is sharp in the case where $A = I_2$, in the sense that there exists a sequence $(\lambda_n)_n$ of positive integers and a sequence $(u_n)_n$ in $L^2(\mathbb T^2)$ such that
\[
    \supp\widehat{u_n}\subset\mathcal S_{\lambda_n},\quad
    \Vert u_n\Vert^4_{L^2(\mathbb T^2)} = 1\quad\text{and}\quad\Vert u_n\Vert^4_{L^4(\mathbb T^2)} \underset{n\rightarrow+\infty}{\longrightarrow}3.
\]
\end{thm}

Let us mention that the sharpness of the constant $C=3$ when $A=I_2$ will not play a role in the other results in this paper, but we take this opportunity to present its proof.

\subsection{The parallelogram lemma} 

Let $\lambda>0$ and $u\in L^2(\twotorus)$ be a function such that $\supp\widehat u\subset\mathcal S_{\lambda}(A)$, and whose Fourier coefficients will be denoted $(a_m)_m$ in the following for simplicity. Let us develop the $L^4$ norm we aim at estimating
\begin{align}\label{eq:ressum}
	\Vert u\Vert_{L^4(\twotorus)}^4
	& = \frac1{(\det A)^2}\int_{\twotorus}\Big\vert \sum_{m\in\mathcal S_{\lambda}(A)} e^{2i\pi m\cdot x} a_m \Big\vert^4\,\mathrm dx \\[5pt]
	& = \frac1{(\det A)^2}\int_{\twotorus} \bigg(\sum_{(m, m')\in\mathcal S_{\lambda}(A)^2}e^{2i\pi(m-m')\cdot x} a_m \overline{a_{m'}}\, \bigg)^2\,\mathrm dx. \nonumber \\[5pt]
	& = \frac1{\det A}\sum_{\substack{(m_1,m_2,m_1',m_2')\in\mathcal S_{\lambda}(A)^4 \\[2pt] m_1 - m_1' + m_2 - m_2' = 0}}a_{m_1}\overline{a_{m_1'}}a_{m_2}\overline{a_{m_2'}}. \nonumber
\end{align}
The above resonant sum can be precisely understood. This is the topic of the following result, which can be found as \cite[Equation (72)]{KKW}, but which will be proven again in the present paper (with quite different arguments) for the sake of completeness.

\begin{lem}[Parallelogram lemma] \label{lem:dia}
Let $m_1, m_1', m_2, m_2' \in \mathbb R^2$ be so that
$$
|m_1| = |m_1'| = |m_2| = |m_2'|\quad\text{and}\quad m_1 -m_1'+m_2 -  m_2' = 0.
$$
Then, we are in one of the following situations:
\begin{enumerate}
    \item[(i)] Parallelogram: $m_1 = m_1'$ and $m_2 = m_2'$.
    \item[(ii)] First degenerate case: $m_1 = - m_2$ and $m_1' = - m_2'$.
    \item[(iii)] Second degenerate case: $m_1 = m_2'$ and $m_2 = m_1'$.
\end{enumerate}
\end{lem}

\begin{proof}
Consider the four points
\begin{align*}
    &P_1 := (0,0), \\
    &P_2 := P_1 + m_1, \\
    &P_3 := P_2 + m_2, \\
    &P_4 := P_3 - m_1', 
\end{align*}
and observe that, because $m_1 -m_1' + m_2 - m_2' = 0$,
$$
P_4 - m_2' = P_1.
$$
\noindent\textbf{$\triangleright$ Case 1: non degenerate case}. Assume that the points $P_1, P_2, P_3, P_4$ are pairwise distinct. Then, the  quadrilateral $P_1P_2P_3P_4$ has edges of lengths $(|m_1|, |m_2|, |m_1'|, |m_2'|)$. As $|m_i| = |m_i'|$, it is therefore either a parallelogram (if it is a simple, \textit{i.e.}~non self-intersecting), or an antiparallelogram (if it is complex, \textit{i.e.}~self-intersecting). However, an antiparallelogram with four edges of same size has necessarily two equal vertices, \textit{i.e.}~is degenerate. Indeed, let us consider e.g.~an antiparallelogram as in Figure \hyperref[fig:antipara]{4}, with the additional property that its four sides are equal in length. Since $\vert m_1\vert = \vert m_2'\vert$ and $\vert m_2\vert = \vert m_1'\vert$, the triangles $P_1P_4P_2$ and $P_3P_4P_2$ are both isosceles, in $P_1$ and $P_3$ respectively. Since they share the side $P_4P_2$, this implies that the two points $P_1$ and $P_3$ are aligned with the middle of $P_4P_2$. Then, the fact that $\vert m_1\vert = \vert m_2\vert$ implies that $P_1$ = $P_3$, and the quadrilateral $P_1P_2P_3P_4$ is therefore degenerate. As we assumed that $P_1$, $P_2$, $P_3$, $P_4$ are pairwise distinct, it follows that the quadrilateral $P_1P_2P_3P_4$ is necessarily a parallelogram and therefore $m_1 = m_1'$ and  $m_2 = m_2'$.

\medskip

\noindent\textbf{$\triangleright$ Case 2: degenerate case}. Assume that there is $i_1\neq i_2$ so that $P_{i_1} = P_{i_2}$. There are six cases: 
\begin{enumerate}
    \item[1.] If  $P_1 = P_2$, then $m_1 = 0 $, from which $m_2 = m_1' = m_2' = 0$.
    \item[2.] If $P_1 = P_3$, then $m_1+m_2 = 0$, from which $m_1' + m_2' = 0$ as well.
    \item[3.] If  $P_1 = P_4$, then $m_2' = 0$, from which $m_1 = m_2 = m_1' = 0$.
    \item[4.] If  $P_2 = P_3$, then $m_2 = 0$, from which $m_1 = m_1' = m_2' = 0$.
    \item[5.] If  $P_2 = P_4$, then $m_2 - m_1' = 0$ and $m_1 - m_2' = 0$. 
    \item[6.] If $P_3=P_4$, then $m_1' = 0$, from which $m_1 = m_2 = m_2' = 0$.
\end{enumerate}
All the cases above but the second and the fifth enter the case of the parallelogram for which $m_1 = m_1'$ and $m_2 = m_2'$.
\end{proof}

\begin{figure}\label{fig:para}

\begin{tabular}{cc}
\begin{tikzpicture}[scale=0.7]         
  \coordinate (A1) at (0,0);           
  \coordinate (A2) at (2.828,0);           
  \coordinate (A3) at ($(A2)+(2,2)$);   
  \coordinate (A4) at ($(A1)+(2,2)$);   

  \draw[-Stealth]  (A1) -- (A2) node[midway, below] {$m_1$};
  \draw[-Stealth] (A2) -- (A3) node[midway, right] {\hspace{0.1pt} $m_2$}; 
  \draw[-Stealth] (A3) -- (A4) node[midway, above] {$-m_1'$};
   \draw[-Stealth] (A4) -- (A1) node[midway, left] {\hspace{-40pt} $-m_2'$};
   
   \node[below left] at (A1) {$P_1$};
   \node[below right] at (A2) {$P_2$};
   \node[above right] at (A3) {$P_3$};
   \node[above left] at (A4) {$P_4$};
         
\end{tikzpicture} 
&
\begin{tikzpicture}[scale=0.7]         
  \coordinate (A1) at (0,0);           
  \coordinate (A2) at ($(A1)+(2,3)$);           

  \draw[<->, >=Stealth]  (A1) -- (A2) node[midway, right] {$m_1 = m_1'$, $m_2 = m_2' = 0$};
   
   \node[below] at (A1) {$P_1=P_4$};
   \node[below right] at (A2) {$P_2=P_3$};
         
\end{tikzpicture}
\\
\begin{tikzpicture}[scale=0.7]         
  \coordinate (A1) at (1,0);           
  \coordinate (A2) at ($(A1)+(3,-2)$);           
  \coordinate (A3) at ($(A2) + (3.391, 1.21)$);   

  \draw[<->, >=Stealth]  (A1) -- (A2) node[midway, left] {\hspace{-65pt} $m_1 =  -m_2$};
  \draw[<->, >=Stealth] (A2) -- (A3) node[midway, right] {\hspace{10pt} $m_2' = -m_1'$}; 
   
   \node[below left] at (A1) {$P_2$};
   \node[below left] at (A2) {$P_1=P_3$};
   \node[above right] at (A3) {$P_4$};
         
\end{tikzpicture}
&
\begin{tikzpicture}[scale=0.7]         
  \coordinate (A1) at (2,0);           
  \coordinate (A2) at ($(A1)+(-2,3)$);           
  \coordinate (A3) at ($(A2) + (3.391, 1.21)$);   

  \draw[<->, >=Stealth]  (A1) -- (A2) node[midway, left] {$m_1 = m_2'$};
  \draw[<->, >=Stealth] (A2) -- (A3) node[midway, right] {\hspace{5pt} $m_2 = m_1'$}; 
   
   \node[below left] at (A1) {$P_1$};
   \node[above left] at (A2) {$P_2=P_4$};
   \node[above right] at (A3) {$P_3$};
         
\end{tikzpicture}
\end{tabular}
\caption{
Examples of the three cases in Lemma \ref{lem:dia}. On the top are two examples of the parallelogram case, corresponding to a non-degenerate parallelogram on the left, and a degenerate one on the right. Bottom left is an example of the first degenerate case, bottom right of the second degenerate case.}
\end{figure}
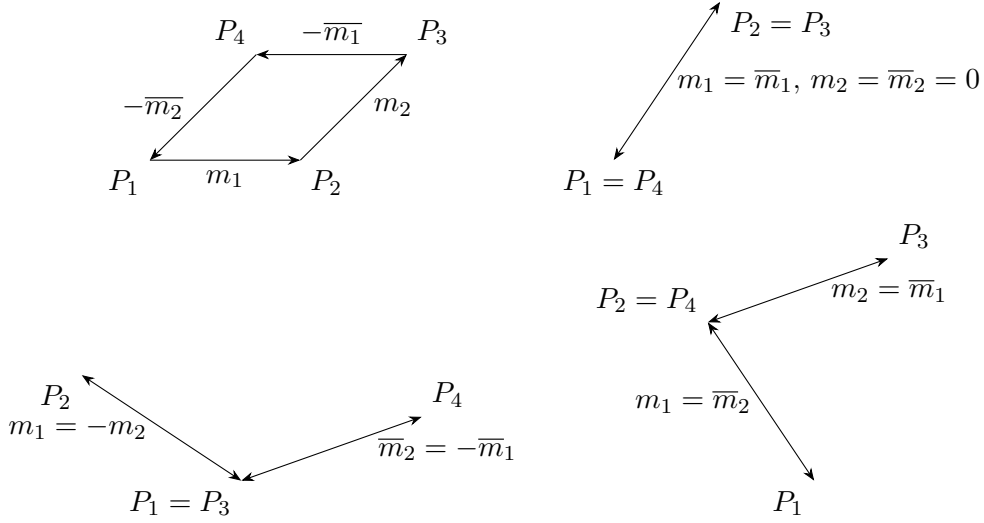

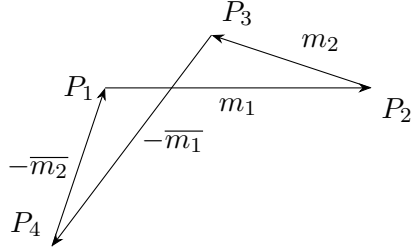
\begin{figure}\label{fig:antipara}

\begin{tikzpicture}[scale=0.7]         
  \coordinate (A1) at (0,0);           
  \coordinate (A2) at ($(A1)+(5,0)$);
  \coordinate (A3) at ($(A2)+(-3,1)$);         
  \coordinate (A4) at ($(A3)+(-3,-4)$);      

  \draw[-Stealth]  (A1) -- (A2) node[midway, below] {$m_1$};
  \draw[-Stealth] (A2) -- (A3) node[midway, above right] {$m_2$}; 
  \draw[-Stealth] (A3) -- (A4) node[midway, right] {$-m_1'$};
  \draw[-Stealth] (A4) -- (A1) node[midway, left] {$-m_2'$};
   
   \node[ left] at (A1) {$P_1$};
   \node[below right] at (A2) {$P_2$};
   \node[above right] at (A3) {$P_3$};
   \node[above left] at (A4) {$P_4$};
         
\end{tikzpicture}
\caption{A (non-equilateral) non-degenerate antiparallelogram. This case is excluded in the proof of Lemma \ref{lem:dia} as an equilateral parallelogram has necessary two equal vertices, \textit{i.e.}~is degenerate.}
\end{figure}

\subsection{The estimate} 

Let us now prove the first part of Theorem \ref{thm:zyg} as a consequence of the parallelogram lemma. Recalling the equality \eqref{eq:ressum}, we first get from the triangle inequality and Lemma \ref{lem:dia} that
\begin{multline*}
    \sum_{\substack{(m_1,m_2,m_1',m_2') \in \mathcal S_{\lambda}(A)^4 \\[2pt] m_1 -m_1'+m_2 - m_2' = 0}} \vert a_{m_1} \overline{a_{m_1'}}a_{m_2}\overline{a_{m_2'}}\vert
    \le \underbrace{\sum_{(m_1, m_2) \in  \mathcal S_{\lambda}(A)^2}  |a_{m_1} \overline{a_{m_1}}  a_{m_2} \overline{a_{ m_2} }|}_{=:\, S_1} \\[5pt]
    + \underbrace{\sum_{(m_1, m_1') \in \mathcal S_{\lambda}(A)^2}  |a_{m_1} \overline{a_{m_1'}}  a_{-m_1} \overline{a_{-m_1'}}|}_{=:\, S_2}
    + \underbrace{\sum_{(m_1, m_2) \in \mathcal S_{\lambda}(A)^2}  |a_{m_1} \overline{a_{m_2}}  a_{m_2} \overline{a_{m_1} }|}_{=:\, S_3}.
\end{multline*}
The first and last sums $S_1$ and $S_3$ are nothing but $\Vert u\Vert_{L^2(\twotorus)}^4$, and the second sum is controlled by
\[
    S_2 = \bigg(\sum_{m\in\mathcal S_{\lambda}(A)} |a_m a_{-m}| \bigg)^2
    \leq \bigg( \sum_{m\in\mathcal S_{\lambda}(A)} \frac 12\big(|a_m|^2 + |a_{-m}|^2\big) \bigg)^2 =  \Vert u\Vert_{L^2(\twotorus)}^4.
\]
Dividing by $\det A$, we finally obtain
\begin{equation}\label{eq:zygl4}
    \Vert u\Vert^4_{L^4(\twotorus)}\le \frac{3}{\det A}\Vert u\Vert^4_{L^2(\twotorus)},
\end{equation}
as expected.

\subsection{Optimality} We now show that the constant $3$ is optimal in the estimate \eqref{eq:zygl4} when $A = I_2$. To that end, let us consider the sequence of functions $(u_n)_n$ in $L^2(\mathbb T^2)$ defined for every $n\in\mathbb N$ by
\[
    u_n(x) := \frac{1}{\sqrt{\vert\mathcal S_{\lambda_n}\vert}} \sum_{m\in\mathcal S_{\lambda_n}} e^{2i\pi m \cdot x},\quad x\in\mathbb T^2,
\]
where the $\lambda_n>0$ are positive integers that will be chosen later. Notice that $\Vert u_n\Vert_{L^2} = 1$. Moreover, it follows from \eqref{eq:ressum} that
\[
    \Vert u_n\Vert^4_{L^4} = \frac1{\vert\mathcal S_{\lambda_n}\vert^2}\sum_{\substack{(m_1,m_2,m_1',m_2')\in(\mathcal S_{\lambda_n})^4 \\[2pt] m_1 - m_1' + m_2 - m_2' = 0}}1
    = \frac{\vert X_n\vert}{\vert\mathcal S_{\lambda_n}\vert^2},
\]
where the set $X_n$ is defined by
\[
    X_n = \big\{(m_1,m_2,m_1',m_2')\in(\mathcal S_{\lambda_n})^4 : m_1 - m_1' + m_2 - m_2' = 0\big\}.
\]
Let us introduce three new sets $X^1_n$, $X^2_n$ and $X^3_n$, corresponding to the three cases in Lemma~\ref{lem:dia}, respectively given by
\begin{align*}
    & X^1_n = \big\{(m_1,m_2,m_1',m_2')\in(\mathcal S_{\lambda_n})^4 : m_1 = m_1',\quad m_2 = m_2'\big\}, \\[5pt]
    & X^2_n = \big\{(m_1,m_2,m_1',m_2')\in(\mathcal S_{\lambda_n})^4 : m_1 = -m_2,\quad m_1' = -m_2'\big\}, \\[5pt]
    & X^3_n = \big\{(m_1,m_2,m_1',m_2')\in(\mathcal S_{\lambda_n})^4 : m_1 = m_2',\quad m_2 = m_1'\big\}.
\end{align*}
Observe that we have the following overlapping decomposition
\[
    X_n = X^1_n \cup X^2_n \cup X^3_n.
\]
Indeed, the direct inclusion follows from Lemma \ref{lem:dia}, while the converse one is immediate. The inclusion-exclusion principle then allows to write that
\[
    \Vert u_n\Vert^4_{L^4} = \frac{\vert X^1_n\vert + \vert X^2_n\vert + \vert X^3_n\vert - \vert X^1_n\cap X^2_n\vert - \vert X^1_n\cap X^3_n\vert - \vert X^2_n\cap X^3_n\vert + \vert X^1_n\cap X^2_n\cap X^3_n\vert}{\vert\mathcal S_{\lambda_n}\vert^2}.
\]
We now have to estimate the above cardinalities. First, it follows from the definition of the sets $X^i_n$ that
\[
    \vert X^i_n\vert = \vert\mathcal S_{\lambda_n}\vert^2,\quad i\in\{1,2,3\}.
\]
Since the two-intersections are explicitly given by
\begin{align*}
    & X^1_n\cap X^2_n = \big\{(m_1,-m_1,m_1,-m_1) : m_1\in\mathcal S_{\lambda_n}\big\}, \\[5pt]
    & X^1_n\cap X^3_n = \big\{(m_1,m_1,m_1,m_1) : m_1\in\mathcal S_{\lambda_n}\big\}, \\[5pt]
    & X^2_n\cap X^3_n = \big\{(m_1,-m_1,-m_1,m_1) : m_1\in\mathcal S_{\lambda_n}\big\},
\end{align*}
we also have
\[
    \vert X^i_n\cap X^j_n\vert = \vert\mathcal S_{\lambda_n}\vert,\quad i\ne j\in\{1,2,3\}.
\]
Finally,
\[
    X^1_n\cap X^2_n\cap X^3_n = \big\{(0,0,0,0)\big\} \cap  (S_{\lambda_n})^4 = \emptyset.
\]
As a consequence, the $L^4$ norm of the functions $u_n$ is given by
\begin{equation}\label{eq:normun}
    \Vert u_n\Vert^4_{L^4} = 3 - \frac{3}{\vert\mathcal S_{\lambda_n}\vert}.
\end{equation}
Moreover, a classical result in number theory, which can be found e.g.~in \cite[Section 16.10]{HW}, states that for every positive integer $N\in\mathbb N^*$ admitting the following prime number decomposition
\[
    N = 2^a\prod p_i^{b_i}\prod q_j^{c_j},
\]
where $a\in\mathbb N$, $b_i, c_j\in\mathbb N$ being almost all equal to $0$, $p_i\equiv 1\, \text{(mod $4$)}$ and $q_j\equiv 3\, \text{(mod $4$)}$ being prime numbers, then
\[
    \vert\mathcal S_{N}\vert = \begin{cases}
        4\prod(b_i+1) & \text{when all the $c_j$ are even,} \\[5pt]
        0 & \text{otherwise}.
    \end{cases}
\]
This result motivates to make the choice $\lambda_n = 5^n$, for which
\[
    \vert\mathcal S_{\lambda_n}\vert = 4(n+1) \underset{n\rightarrow+\infty}{\longrightarrow}+\infty.
\]
Plugging this expression into \eqref{eq:normun}, we obtain
\[
    \Vert u_n\Vert^4_{L^4} = 3 - \frac{3}{4(n+1)} \underset{n\rightarrow+\infty}{\longrightarrow} 3.
\]
Recalling that $\Vert u_n\Vert_{L^2} = 1$, this proves that the constant $3$ is optimal in Zygmund inequality~\eqref{eq:zygl4}.

\begin{rk} Notice that Cauchy-Schwarz inequality immediately implies that for every energy level $\lambda>0$ and for every function $u\in L^2(\mathbb T^2)$ such that $\supp\widehat{u}\subset\mathcal S_{\lambda}$,
\[
    \Vert u\Vert_{L^{\infty}(\mathbb T^2)}\le \vert\mathcal S_{\lambda}\vert^{\frac{1}{2}}\Vert u\Vert_{L^2(\mathbb T^2)}.
\]
By interpolating this estimate and Zygmund inequality from Theorem \ref{thm:zyg}, we therefore deduce that for every $p>4$,
\[
    \Vert u\Vert^p_{L^p(\mathbb T^2)}\le 3\vert\mathcal S_{\lambda}\vert^{\frac{p-4}{2}}\Vert u\Vert^p_{L^2(\mathbb T^2)}.
\]
However, this inequality is not as precise as results known in the literature in the same topic, among others in the case $p=6$ in \cite{BB, KKW}, obtained by more sophisticated approaches and technology.
\end{rk}

\section{Observability estimates on twisted rational tori}\label{sec:obstori}

This last section is devoted to the proof of observability estimates for Schr\"odinger equations posed on the twisted rational tori $\twotorus$ defined in Section \ref{ss:diamond}. By rational tori, we mean that we make the following assumption on the eigenvalues $\lambda_m>0$ of the Laplacian $-\Delta$ acting on $L^2(\twotorus)$, given by \eqref{eq:eigenvalues},
\begin{equation}\label{eq:rational}
    \exists\gamma>0, \forall m\in\mathbb Z^2_A,\quad \lambda_m\in\gamma\mathbb N.
\end{equation}
For example, when $A = I_2$ is the identity matrix, we have $\gamma = 4\pi^2$. Another interesting example in regards of the triangle $\mathcal T$ is the following one.

\begin{ex}\label{ex:mattorus} The torus $\mathbb T^2[\omega]$ defined in \eqref{eq:partitorus} is associated with the following matrix
\[
    A = 3\begin{pmatrix}
        1 & 1/2 \\[5pt]
        0 & \sqrt{3}/{2}
    \end{pmatrix}\in\GL_2(\mathbb R).
\]
Since the eigenvalues of the Laplacian $-\Delta$ on $L^2(\mathbb T^2[\omega])$ are given from \eqref{eq:eigenvalues} by
\[
    \lambda_m = \frac{16\pi^2}{27}(m_1^2 - m_1m_2 + m_2^2),\quad m = (m_1,m_2)\in\mathbb Z^2,
\]
the associated quantity $\gamma$ as in \eqref{eq:rational} is given by $\gamma = 16\pi^2/27$. Notice that these eigenvalues include the eigenvalues \eqref{eq:eigentri} of the Laplacian $-\Delta$ acting on $L^2(\mathcal T)$, which is not surprising due to the correspondance between the torus $\mathbb T^2[\omega]$ and the equilateral triangle $\mathcal T$ presented in Section \ref{subsec:tritotor}.
\end{ex}

In this setting, the goal of this section is to prove the following observability estimates.

\begin{thm}\label{thm:obs} Let $T>0$ be a positive time and $a\in L^2(\twotorus)\setminus\{0\}$ be a non-negative function. There exists a positive constant $C_{a,T}>0$ such that for every $u_0\in L^2(\twotorus)$,
\[
	\Vert u_0\Vert^2_{L^2(\twotorus)}\le C_{a,T}\int_0^T\int_{\twotorus}a(x)\big\vert(e^{it\Delta}u_0)(x)\big\vert^2\,\mathrm dx\,\mathrm dt.
\]
\end{thm}

\subsection{Strichartz estimates on twisted rational tori} As an almost direct consequence of Zygmund inequality stated in Theorem \ref{thm:zyg}, we get $L^4_xL^2_t$ Strichartz estimates without loss of derivatives on twisted rational tori, generalizing the now well-known rectangular case.

\begin{prop}\label{prop:striinhomo} 
For every solution $u$ of the equation
$$
\begin{cases}
	i\partial_t u(t,x) + \Delta u(t,x) = F(t,x),\quad (t,x)\in\mathbb R\times\twotorus, \\
	u(0,\cdot) = u_0\in L^2(\twotorus),
\end{cases}
$$
the following estimate holds
\[
	\Vert u\Vert_{L^4(\twotorus,L^2(0, \frac{2\pi}{\gamma}))}
	\le\bigg(\frac{12\pi^2}{\gamma^2\det A}\bigg)^{\frac14}\big(\Vert u_0\Vert_{L^2(\twotorus)} + \Vert F\Vert_{L^1((0,\frac{2\pi}{\gamma}),L^2(\twotorus))}\big).
\]
\end{prop}

\begin{proof}
Using Duhamel formula, it is straightforward to deduce Proposition \ref{prop:striinhomo} from the corresponding homogeneous estimate: for any $u_0 \in L^2(\twotorus)$,
$$
\big\Vert e^{it\Delta} u_0 \big\Vert_{L^4(\twotorus,L^2(0, \frac{2\pi}{\gamma}))}	\le\bigg(\frac{12\pi^2}{\gamma^2\det A}\bigg)^{\frac14}\Vert u_0\Vert_{L^2(\twotorus)}.
$$
We will therefore show the above. Developing the $L^4_x L^2_t$ norm we aim at estimating, we get
\[
	\big\Vert e^{it\Delta}u_0\big\Vert^4_{L^4_xL^2_t} 
	= \bigg(\frac{2\pi}{\gamma}\bigg)^2\int_{\twotorus} \bigg( \sum_{\lambda>0} |(\Pi_\lambda u_0) (x)|^2\bigg)^2 \,\mathrm dx,
\]
where $\Pi_\lambda$ is the projection on Fourier coefficients of size $\lambda$:
\[
    (\Pi_\lambda u_0)(x) := \frac1{\sqrt{\det A}}\sum_{m\in\mathcal S_{\lambda/(2\pi)^2}(A)} a_m e^{2i\pi m\cdot x},\quad x\in\twotorus.
\] 
Then, by applying Cauchy-Schwarz inequality and Zygmund inequality from Theorem~\ref{thm:zyg}, we get that
\begin{align*}
    \big\Vert e^{it\Delta}u_0\big\Vert^4_{L^4_xL^2_t} 
    & \le \bigg(\frac{2\pi}{\gamma}\bigg)^2\sum_{\lambda,\mu>0}\Vert\Pi_{\lambda}u_0\Vert^2_{L^4}\Vert\Pi_{\mu}u_0\Vert^2_{L^4} \\[5pt]
    & \le \bigg(\frac{2\pi}{\gamma}\bigg)^2\frac{3}{\det A}
    \sum_{\lambda,\mu>0}\Vert\Pi_{\lambda}u_0\Vert^2_{L^2}\Vert\Pi_{\mu}u_0\Vert^2_{L^2} \\[5pt]
    & = \frac{12\pi^2}{\gamma^2\det A}\Vert u_0\Vert^4_{L^2},
\end{align*}
as expected.
\end{proof}

\begin{rk}\label{rk:striinhomo}
Considering the particular case of the matrix $A$ in Example \ref{ex:mattorus} associated with the torus $\mathbb T^2[\omega]$, for which we recall that 
\[
	\det A = \frac{9\sqrt{3}}{2}\quad\text{and}\quad \gamma = \frac{16\pi^2}{27},
\]
and invoking the correspondence between $\mathbb T^2[\omega]$ and the equilateral triangle $\mathcal T$ presented in Section \ref{subsec:tritotor}, we deduce the Strichartz estimates \eqref{eq:inhomostritri} from Proposition \ref{prop:striinhomo}.
\end{rk}

\subsection{Observability}
In the setting of rectangular tori (rational or not), \textit{i.e.}~in the situation where $A$ is symmetric, it is now well-known how to derive the observability estimate in Theorem \ref{thm:obs} with the homogeneous Strichartz estimates from Proposition \ref{prop:striinhomo} at hand, see e.g.~\cite{BZ,LBM}. In this part, we will quickly recall the arguments already appearing in the rectangular case, and present in details what changes in the setting of twisted rational tori. In particular, the rationality assumption \eqref{eq:rational} arises in the dimension reduction, which was not the case for rectangular tori.

Thanks to a Littlewood--Paley decomposition and the Bardos--Lebeau--Rauch uniqueness-compactness argument originated from \cite{BLR} (a quantitative version of which can be found in \cite[Theorem 4]{BBZ}), it is sufficient to prove the following semiclassical observability estimate for every $0<h<h_0$,
\begin{equation}\label{eq:obstrunc}
    \Vert \Pi_{h,\rho}u_0\Vert^2_{L^2(\twotorus)}\le C_{a,T}\int_0^T\int_{\twotorus}
    a(x)\big\vert(e^{it\Delta}\Pi_{h,\rho}u_0)(x)\big\vert^2\,\mathrm dx\,\mathrm dt,
\end{equation}
the frequency localizing operators $\Pi_{h,\rho}$ being given by
\[
    \Pi_{h,\rho} = \chi\bigg(\frac{-h^2\Delta-1}{\rho}\bigg),
\]
where $\chi\in C^{\infty}_0((-1,1))$ is equal to $1$ near zero. Mutatis-mutandis, one can follow the strategy of \cite{BZ}. The estimate \eqref{eq:obstrunc} is proven by contradiction: assuming that the estimate fails along a subsequence $(u_n)_{n\geq 1}$, for $h_n, \rho_n \to 0$, one can extract a defect measure $\mu$ on $\mathbb R_t\times(T^*\twotorus)_{z, \zeta}$ from $(u_n)_{n\geq 1}$. Recall that the measure $\mu$ satisfies the following three properties
\begin{align} 
    & \mu((t_0,t_1)\times T^*\twotorus) = t_1-t_0, \label{eq:contradict_me} \\[5pt]
    & \Supp\mu\subset\big\{(t,z,\zeta)\in\mathbb R\times\twotorus\times\mathbb R^2 : \vert\zeta\vert = 1\big\}, \\
    & \partial_s\int_{\mathbb R}\int_{T^*\twotorus}\varphi(t)\Gamma(z+s\zeta,\zeta)\,\mathrm d\mu = 0,\quad \varphi\in C^0_c(\mathbb R_t),\, \Gamma\in C^{\infty}((T^*\twotorus)_{z, \zeta}) \label{eq:invflowgeo}.
\end{align}
Additionally, thanks to the Strichartz estimates from Proposition \ref{prop:striinhomo} with $F=0$, for any $\tau \geq 0$ there exists a function $m_\tau \in L^2(\twotorus)$ such that for all $f\in L^2(\twotorus)$,
\[
    \int_0^\tau \int_{T^* \twotorus} f(z)\,\mathrm d\mu(t, z, \zeta) = \int_{\twotorus} m_\tau(z) f(z)\,{\rm d}z.
\]
From the setup of the contradiction argument and the definition of defect measures,
$$
\int_{\twotorus} a(z) m_T(z) \, {\rm d}z = 0.
$$
We will show that $m_T$ is identically zero, which is in contradiction with (\ref{eq:contradict_me}). The idea is that $\mu$ has most of its mass on rational directions. But the flow from such directions is periodic, which permits to perform a reduction to a one-dimensional torus, for which the observability properties are well known. Let us mention that in our geometric context, \textit{i.e.}~in the twisted torus $\twotorus$, one has to change the definition of rational directions compared to what is done in \cite{BZ}. Precisely, we consider
\[
    W^n = \bigg\{(z,\zeta)\in T^*\twotorus : \zeta = \frac{Am}{\vert Am\vert},\,m\in\mathbb Z^2,\,\max(\vert m_1\vert,\vert m_2\vert)\le n,\,\gcd(m_1,m_2) = 1\bigg\}.
\]
By using the same arguments as in \cite[Lemma 2.2]{BZ}, which rely on the unique ergodicity of the flow $z\mapsto z+s\zeta$, we get that
\[
    \forall\varepsilon>0,\;\exists n\geq0,\quad\widetilde\mu_T(W^n)<\varepsilon,
\]
where $\widetilde\mu_T$ is the defect measure $\mu$ integrated in time on $[0,T]$, \textit{i.e.}
\[
    \int_{T^*\twotorus}\Gamma(z,\zeta)\,\mathrm d\widetilde{\mu}_T(z,\zeta) 
    = \int_0^T\int_{\twotorus}\Gamma(z,\zeta)\,\mathrm d\mu(t,z,\zeta),\quad \Gamma\in C^{\infty}_c(T^*\twotorus).
\]
Let us fix $n_0\geq0$ large enough such that
\[
    \widetilde\mu_T(T^*\twotorus\setminus W^{n_0})<T,
\]
and also
\[
    \zeta_0\in\big\{\zeta\in\mathbb R^2 : \exists z\in\twotorus,\,(z,\zeta)\in\supp(\widetilde{\mu_T}\vert_{W^{n_0}})\big\}.
\]
Mimicking \cite[Lemma 2.3]{BZ}, we have for some $F\in L^2(\twotorus)\setminus\{0\}$ such that $F\geq0$,
\begin{equation}\label{eq:reducdim}
    \widetilde{\mu_T}\big\vert_{\twotorus\times\{\zeta_0\}} = F\otimes\delta_{\zeta = \zeta_0}.
\end{equation}

We now set
\[
    e_1 := {A\begin{pmatrix}
        1 \\
        0
    \end{pmatrix}}
    \quad\text{and}\quad
   e_2 := {A\begin{pmatrix}
        0 \\
        1
    \end{pmatrix}},
\]
and show that up to changing the torus $\twotorus$, we can actually assume that $\zeta_0= \frac{e_2}{|e_2|}$. Indeed, it suffices to consider the following rotation, where the vectors $\zeta_0$ and $e_2$ are considered as elements of $\mathbb C^*$
\[
    \phi(z) := \zeta_0 |e_2|e_2^{-1}z,\quad z\in\twotorus,
\]
to work with the new functions $\widetilde u_n := u_n\circ\phi$, and  to change the torus $\twotorus$ to
\[
    \mathbb T^2_{rA} := \mathbb R^2/(rA\mathbb Z^2),
\]
where $r := q_0 |e_2|^{-1}\vert Am\vert$ with $q_0 \in \mathbb N^*$ chosen later, and $m\in\mathbb Z^2$ given by
\[ 
    \zeta_0 = \frac{Am}{\vert Am\vert}\in W^{m_0}.
\]
 Let us show that for a well chosen $q_0 \in \mathbb N^*$, the functions $\widetilde u_n$ are well-defined on this new torus $\mathbb T^2_{rA}$. To that end, let us consider a point $z\in\mathbb T^2_{rA}$ and an integer $p\in\mathbb Z$. On the one hand, the periodicity in the direction $re_2$ is immediate: indeed, since $u_n\in L^2(\twotorus)$ and $pq_0 Am \in A\mathbb Z^2$,
\begin{align*}
    \widetilde u_n(z + pre_2) & = u_n(\zeta_0 |e_2|e_2^{-1}z + \zeta_0 |e_2|e_2^{-1}pre_2)
    = u_n(\zeta_0 |e_2|e_2^{-1}z + pq_0Am) \\[5pt]
    & = u_n(\zeta_0 |e_2|e_2^{-1}z) = \widetilde u_n(z).
\end{align*}
On the other hand, the periodicity in the direction $re_1$ will be a consequence of the rationality assumption \eqref{eq:rational} for a good choice of $q_0 \in \mathbb N^*$. Indeed
\begin{align*}
	\widetilde u_n(z + pre_1) & = u_n(\zeta_0 |e_2|e_2^{-1}z + \zeta_0 |e_2|e_2^{-1}pre_1) = u_n(\zeta_0 |e_2|e_2^{-1}z + pq_0 Am e_2^{-1}e_1) \\[5pt]
	&= u_n(\zeta_0 |e_2|e_2^{-1}z + pq_0 m_1 e_2^{-1}e_1^2 + pq_0 m_2 e_1) =  u_n(\zeta_0 |e_2|e_2^{-1}z + pq_0 m_1 e_2^{-1}e_1^2),
\end{align*}
where we denoted $m=(m_1,m_2)$ and used the periodicity of $u_n$ in the direction $e_1$ in the second line. To conclude, we will show that we can choose $q_0 \in \mathbb N^*$ so that
\begin{equation} \label{eq:rat_rot}
	q_0 e_2 ^{-1} e_1^2 \in \mathbb Z e_1 + \mathbb Z e_2,
\end{equation}
which gives the periodicity of $\widetilde u_n$ in the direction $re_1$ thanks to the above and the fact that $u_n \in L^2(\mathbb T_A)$. To this effect, denote
$$
A^{-1} := 
\begin{pmatrix}
a & b \\ c & d
\end{pmatrix}.
$$
Working with the complex representation, a direct computation shows
$$
e_2 ^{-1} e_1^2 = \frac{1}{(a^2 + b^2)(ad-bc)} \begin{pmatrix}
-d^2b + c^2b -2adc\\ -d^2a+c^2a +2bcd
\end{pmatrix}.
$$
We now use the change of base formula 
\[
    \alpha \begin{pmatrix}
        1 \\
        0
    \end{pmatrix} + \beta \begin{pmatrix}
        0 \\
        1
    \end{pmatrix} = Xe_1 + Ye_2\ \Longleftrightarrow\ \begin{pmatrix}
        X \\
        Y
    \end{pmatrix} = A^{-1}\begin{pmatrix}
        \alpha \\
        \beta
    \end{pmatrix},
\]
to get
$$
e_2 ^{-1} e_1^2 = - \frac{1}{a^2+b^2}\big( 2(ac + bd)e_1 + (c^2+d^2)e_2 \big).
$$
We recognize the coefficients of $A^{-1}(A^{-1})^*$:
$$
A^{-1}(A^{-1})^* = 
\begin{pmatrix}
a^2 + b^2 & ac+bd\\  ac+bd & c^2 + d^2
\end{pmatrix}.
$$
To relate these coefficients to the rationality assumption \eqref{eq:rational}, it will be useful to relabel the eigenvalues $\lambda_m$ of the Laplace operator $-\Delta$ acting on $L^2(\twotorus)$, by stating that they are given by
\[
	\lambda_{n_1,n_2} = \left\vert 2\pi(A^{-1})^*\begin{pmatrix}
        n_1 \\
        n_2
    \end{pmatrix}\right\vert^2,\quad \begin{pmatrix}
        n_1 \\
        n_2
    \end{pmatrix}\in\mathbb Z^2.
\] 
With this new terminology,
\begin{equation} \label{eq:quad_form}
A^{-1}(A^{-1})^* = \frac1{16\pi^4}
\begin{pmatrix}
\lambda_{1,0} & \frac 12 (\lambda_{1,1} - \lambda_{1,0}-\lambda_{0,1})\\ \frac 12 (\lambda_{1,1} - \lambda_{1,0}-\lambda_{0,1}) & \lambda_{0,1}
\end{pmatrix}.
\end{equation}
It follows that, under the rationality assumption \eqref{eq:rational},
$$
e_2 ^{-1} e_1^2 = \frac{p_1}{q_1} e_1 + \frac{p_2}{q_2} e_2, \quad (p_1, q_1), (p_2, q_2) \in \mathbb Z \times \mathbb N^*.
$$
Taking $q_0 := \operatorname{lcm}(q_1, q_2)$, we obtain (\ref{eq:rat_rot}), and hence $\widetilde u_n \in  L^2(\mathbb T^2_{rA})$. Observe, crucially, that this modification of torus does not change the rationality assumption \eqref{eq:rational}, which is still valid on $\mathbb T^2_{rA}$; and that $\widetilde u_n$ is still solution of the Schr\"odinger equation because the Laplacian $\Delta$ commutes with rotations. We now go back to the notations $u_n$ and $\mathbb T^2_A$ and assume that $\zeta_0= \frac{e_2}{|e_2|}$.

We are finally in position to proceed to the dimension reduction. By invariance of the defect measure $\mu$ with respect to the geodesic flow (property \eqref{eq:invflowgeo}) and the decomposition \eqref{eq:reducdim}, we have
\[
    \widetilde{\mu_T}\big\vert_{{\mathbb T}^2_{A}\times\{\frac{e_2}{|e_2|}\}} = g(x)\,\mathrm dx\mathrm dy\otimes\delta_0(\xi)\otimes\delta_1(\eta),
\]
where we work with the coordinates $(x,y)$ relative to the basis $(e_1,e_2)$. To perform explicitly the dimension reduction, we write the Laplacian in this basis. Precisely, let us make the following change of function
\[
	v_n(t,x,y) := u_n\left(t,A\begin{pmatrix}
        x \\
        y
    \end{pmatrix}\right),\quad (t,x,y)\in\mathbb R\times\mathbb T^2.
\]
A straightforward computation shows that the functions $v_n$ are solutions of the following twisted Schr\"odinger equation
\[
	i\partial_tv_n + \divv(A^{-1}(A^*)^{-1}\nabla v_n) = 0\quad\text{in $\mathbb R\times\mathbb T^2$.}
\]
Let $\psi_A$ be the quadratic form associated with $A^{-1}(A^*)^{-1}$:
$$
\psi_A(\Xi) := A^{-1}(A^*)^{-1} \Xi^* \cdot \Xi^*, \quad \Xi \in \mathbb R^2.
$$
Denoting by $v^k_n(t,\cdot)$ the partial Fourier coefficients of the functions $v_n(t,\cdot)$ with respect to the variable $y\in\mathbb T$, with $k\in 2\pi\mathbb Z$, we deduce by taking the Fourier transform in the above equation that the functions $v^k_n$ satisfy
\begin{equation}\label{eq:eqnfour}
	i\partial_tv^k_n + \psi_A(\partial_x, -ik)v^k_n = 0\quad\text{in $\mathbb R\times\mathbb T$.}
\end{equation}
We will now show that, under the rationality assumption \eqref{eq:rational}, $v^k_n$ is in fact solution to a straight one-dimensional Schr\"odinger equation up to conjugation by an oscillating factor. To this end, let us perform another change of function and consider
\begin{equation}\label{eq:w_conj}
	w^k_n(t,x) := v^k_n(t,x)e^{i\alpha kx + i\beta k^2t},\quad (t,x)\in\mathbb R\times(\mathbb R/\mu\mathbb Z),
\end{equation}
where we aim at choosing the constants $\alpha,\beta, \Lambda\in\mathbb R$ and $\mu \in \mathbb N^*$ so that $w^k_n$ satisfies the following free Schr\"odinger equation
\begin{equation} \label{eq:conj_free}
	i\partial_t w^k_n + \Lambda\partial^2_x w^k_n = 0\quad\text{in $\mathbb R\times(\mathbb R/\mu\mathbb Z)$}.
\end{equation}
Observe that, from (\ref{eq:quad_form}), the quadratic form $\psi_A$ is given by
$$
\psi_A(\xi,\eta) = \frac1{4\pi^2}\big(\lambda_{1,0}\xi^2 + (\lambda_{1,1} - \lambda_{0,1} - \lambda_{1,0}) \xi\eta +\lambda_{0,1}\eta^2\big),\quad(\xi,\eta)\in\mathbb R^2.
$$
Using this expression for $\psi_A$, injecting (\ref{eq:w_conj}) into (\ref{eq:eqnfour}) yields to, in order for $w_n^k$ to be solution to (\ref{eq:conj_free})
\begin{align*}
	\Lambda = \frac{\lambda_{1,0}}{4\pi^2},\quad \alpha = \frac{\lambda_{0,1} + \lambda_{1,0} - \lambda_{1,1}}{2\lambda_{1,0}}
	\quad\text{and}\quad\beta = \frac{\lambda_{0,1}}{4\pi^2} - \frac{(\lambda_{0,1}+\lambda_{1,0}-\lambda_{1,1})^2}{16\pi^2\lambda_{1,0}}.
\end{align*}
At that point, it is fundamental to notice that $\alpha\in\mathbb Q$ as a consequence of the rationality assumption \eqref{eq:rational} on the twisted torus $\twotorus$. Therefore, $\alpha = p/q$ with $(p,q) \in \mathbb Z \times \mathbb N^*$, and we take $\mu := q$. Having fixed such $\Lambda$, $\alpha$, $\beta$, $\mu$, $w_n^k$ is $\mu=q$--periodic and hence well defined on $\mathbb R\times(\mathbb R/\mu\mathbb Z)$, and verifies (\ref{eq:conj_free}). Before going on, let us comment that in the rectangular case, \textit{i.e.}~when $A$ is symmetric, up to a rotation and a change of function, the corresponding quadratic form has no cross-term and so $\alpha = 0$, and no rationality assumption was needed at this point. Coming back to $v_n^k$, we get
$$
v_n^k(t,x) = e^{-i\alpha k x} e^{i\Lambda t \partial_x^2}\big( e^{i\alpha k x} v_n^k(0) \big) e^{-i\beta k^2 t},
$$
where $(e^{i\Lambda t \partial_x^2})_{t\in\mathbb R}$ denotes the Schr\"odinger flow on $\mathbb R/\mu\mathbb Z$. We can now write
$$
    (e^{it\Delta}v_n)(x,y) = \sum_{k\in2\pi\mathbb Z} e^{-i\alpha k x} e^{i\Lambda t \partial_x^2}\big( e^{i\alpha k x} v_n^k(0) \big) e^{-i\beta k^2 t} e^{iky}.
$$
Using the one-dimensional observability result \cite[Lemma 2.4]{BZ}, combined with an approximation argument involving the homogeneous estimate from Proposition \ref{prop:striinhomo}, see e.g. \cite[pp. 337 - 338]{BZ}, then permits to get a contradiction and to conclude the proof of Theorem \ref{thm:obs}.

\subsection{Conclusion} As observed in the end of Section \ref{sec:til}, Theorem \ref{thm:obs} combined with the correspondence between the equilateral triangle $\mathcal T$ and the twisted torus $\tori$ stated as Corollary \ref{cor:red_twist}, give Theorem \ref{thm:obstriangle}.

\appendix

\section{Orthogonality of eigenfunctions}\label{app:ortho}

The purpose of this appendix is to show that both families of eigenfunctions $(u_{m,n})_{(m,n)}$ and $(v_{m,n})_{(m,n)}$, defined by \eqref{eq:ef}, are orthogonal in $L^2(\mathcal T)$. Hence, we show

\begin{prop} \label{prop:ortho}Let $(m_1, n_1), (m_2, n_2) \in \mathbb Z^2$ be so that $3 \,|\, m_1 + n_1$ and $3 \,|\, m_2+n_2$.
If $(m_1, n_1) \not\sim (m_2, n_2)$, i.e. $G k_{m_1, n_1} \neq G k_{m_2, n_2}$, then
$$
\int_{\mathcal T} u_{m_1, n_1} \overline {u_{m_2, n_2}} = 0 \quad \text{ and }\quad \int_{\mathcal T} v_{m_1, n_1} \overline {v_{m_2, n_2}} = 0.
$$
\end{prop}

The strategy is to exploit anew the decomposition \eqref{eq:pH}, and to deduce this result from an orthogonality property on the hexagon $\mathcal H_1$.  
Let us start with a result on the exponential functions on which the functions $u_{m,n}$ and $v_{m,n}$ are constructed.

\begin{lem} \label{lem:hexagone} Let $(m_1, n_1), (m_2, n_2) \in \mathbb Z^2$ be so that $3 \,|\, m_1 + n_1$ and $3 \,|\, m_2+n_2$. Then
$$
\int_{{\mathcal H}_1}\exp\big(i(k_{m_1, n_1} - k_{m_2, n_2})\cdot x\big) =  \frac{3\sqrt{3}}{2} {\mathbbm 1}_{(m_1, n_1) = (m_2, n_2)}.
$$
\end{lem}

\begin{proof}
Since the function $x \mapsto \exp\big(i(k_{m_1, n_1} - k_{m_2, n_2})\cdot x\big)$ is both $3e$ and $3\omega$ periodic, it follows from \eqref{eq:pH} and a change of variables to coordinates $x=a+\omega b$ that
\begin{align*}
	&\ \sum_{i = 1, 2, 3} \int_{\mathcal H_i} \exp\big(i(k_{m_1, n_1} - k_{m_2, n_2})\cdot x\big)  = 
	\int_{\mathcal P} \exp\big(i(k_{m_1, n_1} - k_{m_2, n_2})\cdot x\big) \, \mathrm dx \\
	= &\ \frac{\sqrt 3}{2}\int_{0}^3  \exp\bigg(\frac{2i\pi}{3}(m_1-m_2) a\bigg) \, \mathrm da \int_0^3 \exp\bigg(\frac{2i\pi}{3}(n_1-n_2)b\bigg) \, \mathrm db \\
	= &\ \frac{9\sqrt{3}}{2}{\mathbbm 1}_{(m_1, n_1) = (m_2, n_2)}.
\end{align*}
On the other hand, we have by a change of variable
\begin{align*}
	&\ \sum_{i = 1, 2, 3} \int_{\mathcal H_i} \exp\big(i(k_{m_1, n_1} - k_{m_2, n_2})\cdot x\big) = \sum_{p=0,1,2} \int_{{\mathcal H}_1 + pe + p\omega} \exp\big ( i(k_{m_1,n_1} - k_{m_2,n_2})\cdot x\big) \, \mathrm dx \\
	= &\ \bigg(\sum_{p=0,1,2} \exp\big ( i(k_{m_1,n_1} - k_{m_2,n_2})\cdot (pe + p\omega) \big) \bigg) \int_{{\mathcal H}_1} \exp\big ( i(k_{m_1,n_1} - k_{m_2,n_2})\cdot x\big)  \, \mathrm dx.
\end{align*}
But the above sum can be explicitly computed
$$
\sum_{p=0,1,2} \exp\big ( i(k_{m_1,n_1} - k_{m_2,n_2} )\cdot (pe + p\omega) \big) = \sum_{p=0, 1, 2} \exp{ \bigg( \frac{2i\pi}{3}p(m_1 - m_2 +n_1 - n_2)\bigg)} = 3,
$$
because  $3 \,|\, m_1 -m_2+ n_1-n_2$. The expected result follows.
\end{proof}

We deduce from Lemma \ref{lem:hexagone} the following result, corresponding to Neumann boundary condition.

\begin{lem} \label{lem:res_tri} Let $(m_1,n_1), (m_2, n_2) \in \mathbb Z^2$ be so that $3 \,|\, m_i + n_i$ for $i = 1,2$. Then,
\[
	\sum_{g_1,g_2\in G} \int_{\mathcal T} \exp\big ( i(g_1 k_{m_1,n_1} - g_2 k_{m_2,n_2})\cdot x\big) \, \mathrm dx
	= \frac{3\sqrt{3}}{2} \sum_{g \in G} {\mathbbm 1}_{k_{m_1,n_1} = gk_{m_2,n_2}}.
\]
\end{lem}

\begin{proof} Let us denote by $I$ the quantity we aim at computing. By a change of variable in the integral, then in the sum, we get
\begin{align*}
	I = &\ \sum_{g_1,g_2\in G} \int_{\mathcal T} \exp\big ( i(g_1 k_{m_1,n_1} - g_2 k_{m_2,n_2})\cdot x\big) \, \mathrm dx \\
	= &\  \sum_{g_1,g_2\in G}  \int_{\mathcal T} \exp\big ( i( k_{m_1,n_1} - {g_1^*}^{-1}g_2k_{m_2,n_2})\cdot g_1^* x\big) \, \mathrm dx \\
		= &\  \sum_{g_1,g_2\in G}  \int_{ g_1^* \mathcal T} \exp\big ( i( k_{m_1,n_1} - {g_1^*}^{-1}g_2k_{m_2,n_2})\cdot x\big) \, \mathrm dx \\
		= &\  \sum_{g_1 \in G} \sum_{g_2\in {g_1^*}^{-1}G}  \int_{ g_1^* \mathcal T} \exp\big ( i( k_{m_1,n_1} - g_2k_{m_2,n_2})\cdot x\big) \, \mathrm dx.
\end{align*}
Hence, using the fact that ${g_1^*}^{-1} G = G$ since $g_1^*\in G$, then the decomposition of the hexagon $\mathcal H_1$ given by \eqref{eq:h1} together with the fact that the map $g_1 \in G \mapsto g_1^* \in G$ is a bijection, we get
\begin{align*}
	I = &\ \sum_{g_1,g_2\in G}  \int_{g_1^* \mathcal T} \exp\big ( i( k_{m_1,n_1} - g_2 k_{m_2,n_2})\cdot  x\big) \, \mathrm dx \\
	= &\ \sum_{g_2\in G}  \int_{{\mathcal H}_1} \exp\big ( i( k_{m_1,n_1} -  g_2 k_{m_2,n_2})\cdot  x\big) \, \mathrm dx.
\end{align*}
Lemma \ref{lem:hexagone} finally gives the expected formula.
\end{proof}

The next lemma gives the corresponding alternating  formula for the Dirichlet case.

\begin{lem} \label{lem:res_tri_di} Let $(m_1,n_1), (m_2, n_2) \in \mathbb Z^2$ be so that $3 \,|\, m_i + n_i$ for $i = 1,2$. Then,
\[
	\sum_{g_1,g_2\in G} \varepsilon(g_1) \varepsilon(g_2) \int_{\mathcal T} \exp\big ( i(g_1 k_{m_1,n_1} - g_2 k_{m_2,n_2})\cdot x\big) \, \mathrm dx
	= \frac{3\sqrt{3}}{2} \sum_{g \in G}  \varepsilon(g){\mathbbm 1}_{k_{m_1,n_1} = gk_{m_2,n_2}}.
\]
\end{lem}

\begin{proof}
Denoting by $I$ the left-hand side and following the proof of Lemma \ref{lem:res_tri}, we get
\[
	I = \sum_{g_1,g_2\in G} \varepsilon(g_1)\varepsilon(g_2)
	\int_{ g^*_1 \mathcal T} \exp\big ( i( k_{m_1,n_1} - {g_1^*}^{-1} g_2 k_{m_2,n_2})\cdot  x\big) \, \mathrm dx.
\]
But, since the group $G$ is generated by $R_1$ and $R_2$ with $R_1^* = R_1$ and $R_2^* = R_2$, and as $\varepsilon$ is a group homomorphism to a commutative group, we have $\varepsilon(g^*) = \varepsilon(g)$ for any $g\in G$. Hence, we can write in the above
\[
	\varepsilon(g_1)\varepsilon(g_2)
	= \varepsilon(g^*_1)\varepsilon(g_2)
	= \varepsilon(g^*_1) \varepsilon(g^*_1)\varepsilon({g_1^*}^{-1}g_2)
	= \varepsilon(g^*_1)^2\varepsilon({g_1^*}^{-1}\overline g_2)
	= \varepsilon({g_1^*}^{-1}g_2),
\]
since $\varepsilon$ takes values in $\{-1,1\}$, and the end of proof follows similarly as for Lemma~\ref{lem:res_tri}.
\end{proof}

\begin{proof}[Proof of Proposition \ref{prop:ortho}]
The result directly follows from Lemma \ref{lem:res_tri} and Lemma \ref{lem:res_tri_di}, using the fact that $Gk_{m_1, n_1} \neq Gk_{m_2, n_2} \iff k_{m_1, n_1} \not\in Gk_{m_2, n_2}$.
\end{proof}

\subsection*{Acknowledgments} The authors warmly thank A. Bailleul for pointing out the reference \cite{HW} to them. P. A. is partially supported by the ANR project CHAT ANR-24-CE40-5470. 

\end{document}